\documentclass[11 pt]{amsart}

\usepackage[utf8]{inputenc}
\usepackage[english]{babel}
\usepackage{morefloats, amssymb, amsfonts, amsmath, mathrsfs, amsthm, epsfig, amscd, xcolor, subfigure, hyperref}
\usepackage[all]{xy}
\usepackage{epigraph}
\usepackage{aliascnt}
\usepackage{subfigure}

\newcommand{\Z}{\mathbb{Z}}

\newcommand{\GL}{\mathrm{GL}}

\newtheorem{lemma}{Lemma}[section]

\newaliascnt{teo}{lemma}
\newtheorem{teo}[teo]{Theorem}
\aliascntresetthe{teo}

\newaliascnt{prop}{lemma}
\newtheorem{prop}[prop]{Proposition}
\aliascntresetthe{prop}

\theoremstyle{remark}
\newaliascnt{rem}{lemma}
\newtheorem{rem}[rem]{Remark}
\aliascntresetthe{rem}

\addto\extrasenglish{    }

\usepackage{caption} 
\makeatletter
\newcommand{\addresseshere}{%
  \enddoc@text\let\enddoc@text\relax
}
\makeatother

\usepackage[letterpaper,left=1in,right=1in,top=1in,bottom=1in]{geometry}
\author{Michelle Chu}
\address{Department of Mathematics, University of Illinois at Chicago, 851 S. Morgan Street, Chicago, IL 60607-7045, USA}
\email{michu@uic.edu} 

\author{Alexander Kolpakov}
\address{Institut de math\'{e}matiques, Universit\'{e} de Neuch\^{a}tel, Rue Emile-Argand 11, CH-2000 Neuch\^{a}tel, Suisse / Switzerland}
\email{kolpakov.alexander@gmail.com}

\thanks{M.C. was supported by NSF grant DMS-1803094 and A.K. was supported by SNSF project no.~PP00P2-170560.}

\title[]{A hyperbolic counterpart to Rokhlin's cobordism theorem}

\begin{document}

\begin{abstract}
The purpose of the present paper is to prove existence of \emph{super-exponentially} many compact orientable hyperbolic arithmetic $n$-manifolds that are geometric boundaries of compact orientable hyperbolic $(n+1)$-manifolds, for any $n \geq 2$, thereby establishing that these classes of manifolds have the same growth rate with respect to volume as all compact orientable hyperbolic arithmetic $n$-manifolds. An analogous result holds for non-compact orientable hyperbolic arithmetic $n$-manifolds of finite volume that are geometric boundaries, for $n \geq 2$.
\end{abstract}

\maketitle

\epigraph{\textit{In homage to V. Rokhlin\\ on his 100th anniversary.}}

\medskip

\section{Introduction}\label{sec: Intro}
A classical result by V.~Rokhlin states that every compact orientable $3$-manifold bounds a compact orientable $4$-manifold, and thus the three-dimensional cobordism group is trivial. Rokhlin also proved that a compact orientable $4$-manifold bounds a compact orientable $5$-manifold if and only if its signature is zero, which is true for all closed orientable hyperbolic $4$-manifolds. One can recast the question of bounding in the setting of hyperbolic geometry, which generated plenty of research directions over the past decades. 

A hyperbolic manifold is a manifold endowed with a Riemannian metric of constant sectional curvature $-1$. Throughout the paper, hyperbolic ma\-ni\-folds are assumed to be connected, orientable, complete, and of finite volume, unless otherwise stated. We refer to \cite{MaRe, VS} for the definition of an arithmetic hyperbolic manifold.

A connected hyperbolic $n$-manifold $\mathcal{M}$ is said to \emph{bound geometrically} if it is isometric to $\partial \mathcal{W}$ for a hyperbolic $(n+1)$-manifold $\mathcal{W}$ with totally geodesic boundary.

Indeed, some interest in hyperbolic manifolds that bound geometrically was kindled by the works of Long, Reid \cite{LR1, LR2} and Niemershiem \cite{N}, motivated by a preceding work of Gromov \cite{G1, G2} and a question by Farrell and Zdravkovska \cite{FZ}. This question is also related to hyperbolic instantons, as described in \cite{RT1, RT2}. 

As \cite{LR1} shows many closed hyperbolic $3$-manifolds do not bound geometrically: a necessary condition is that the $\eta$-invariant of the $3$-manifold must be an integer. The first example of a closed hyperbolic $3$-manifold known to bound geometrically was constructed by Ratcliffe and Tschantz in \cite{RT1} and has volume of order $200$. 

The first examples of knot and link complements that bound geometrically were produced by Slavich in \cite{S1, S2}. However, \cite{KRR} implies that there are plenty of cusped hyperbolic $3$-manifolds that cannot bound geometrically, with the obstruction being the geometry of their cusps.

In \cite{LR3}, by using arithmetic techniques, Long and Reid built infinitely many orientable hyperbolic $n$-manifolds $\mathcal{N}$ that bound geometrically an $(n+1)$-manifold $\mathcal{M}$, in every dimension $n\geqslant 2$. Every such manifold $\mathcal{N}$ is obtained as a cover of some $n$-orbifold $O_\mathcal{N}$ geodesically immersed in a suitable $(n+1)$-orbifold $O_\mathcal{M}$. However, this construction gives no control on the volume of the manifolds. 

In \cite{BGLS}, Belolipetsky, Gelander, Lubotzky, and Shalev showed that the growth rate of all orientable arithmetic hyperbolic manifolds, up to isometry, with respect to volume is super-exponential, in all dimensions $n\geq 2$.  Their lower bound used a subgroup counting technique due to Lubotzky \cite{L95}. In the present paper we shall use the ideas of \cite{LR3} together with the subgroup counting argument due to Lubotzky \cite{L95} (also used in \cite{BGLS}), together with the more combinatorial colouring techniques from \cite{KS} in order to prove the following facts:

\begin{prop}\label{prop1}
Let $\kappa_n(x) = $ the number of non-isometric non-orientable compact arithmetic hyperbolic $n$-manifolds of volume $\leq x$. Then we have that $\kappa_n(x) \asymp x^x$ for any $n \geq 3$. 
\end{prop}

\begin{prop}\label{prop2}
Let $\nu_n(x) = $ the number of non-isometric non-orientable cusped arithmetic hyperbolic $n$-manifolds of volume $\leq x$. Then we have that $\nu_n(x) \asymp x^x$ for any $n \geq 3$.
\end{prop}

Above, the notation ``$f(x) \asymp x^x$'' for a function $f(x)$ is a shorthand for ``there exist positive constants $A_1, B_1, A_2, B_2$, and $x_0$, such that $A_1 x^{B_1 x} \leq f(x) \leq A_2 x^{B_2 x}$, for all $x \geq x_0$.''

The techniques of \cite{BGLS, L95} provide us with super-exponentially many manifolds of volume $\leq x$ (for $x$ sufficiently large) by employing a retraction of the manifold's fundamental group into a free group. In our case we need however to take extra care in order to arrange for the kernel of such retraction comprise an orientation-reversing element.  Here Coxeter polytopes and reflection groups come into play as natural sources of orientation-reversing isometries, as well as building blocks for manifolds. 

Then, by using the embedding technique from \cite{KRS} and the techniques for constructing torsion-free subgroups from \cite{LR3} (see \autoref{lem:TF-subgroup}, also \autoref{lem:No-torsion} below), we obtain the following theorems establishing that the growth rate with respect to volume of arithmetic hyperbolic manifolds bounding geometrically is the same as that over all arithmetic hyperbolic manifolds. 

\begin{teo}\label{thm1}
Let $\beta_n(x) = $ the number of non-isometric orientable compact arithmetic hyperbolic $n$-manifolds of volume $\leq x$ that bound geometrically. Then we have that $\beta_n(x) \asymp x^x$ for $n \geq 3$.
\end{teo}

\begin{teo}\label{thm2}
Let $\gamma_n(x) =  $ the number of non-isometric orientable cusped arithmetic hyperbolic $n$-manifolds of volume $\leq x$ that bound geometrically. Then we have that $\gamma_n(x) \asymp x^x$ for $n \geq 3$.
\end{teo}

As a by-product, we provide a different proof to a part of the results in \cite{KR} and construct a few  new Coxeter polytopes not otherwise available on the literature. For dimensions $n = 2, \ldots, 6$ in the compact case, and dimensions $n = 2, \ldots, 13$ in the cusped case, we construct explicit examples of retractions onto free groups. More involved computations may be performed in dimensions $n = 14, 15$ (using the polytopes from \cite{Allcock}) and $n = 18, 19$ (using the polytopes from \cite{KapVin}). However, the general case follows from the main result of Bergeron, Haglund, Wise \cite{BHW} on virtually retractions of arithmetic groups of simplest type onto geometrically finite subgroups. 

It is also worth mentioning that a linear lower bound with respect to volume for the number of isometry classes of compact orientable bounding hyperbolic $3$-manifolds was obtained previously in \cite{MZ} by extending the techniques from \cite{KMT} and comparing the Betti numbers of the resulting manifolds.  

Given the present question's background, one may think of \autoref{thm1} as a ``hyperbolic counterpart'' to Rokhlin's theorem. Indeed, not every compact orientable arithmetic hyperbolic $3$-manifold bounds geometrically, but the number of those that do has the same growth rate as the number of all compact orientable arithmetic hyperbolic $3$-manifolds. In the light of Wang's theorem \cite{Wang} and the results of \cite{BGLM}, an analogous statement can be formulated for geometrically bounding hyperbolic $4$-manifolds without arithmeticity assumption. 

As for the closed hyperbolic surfaces that bound geometrically, it follows from the work of Brooks \cite{Brooks} that for each genus $g \geq 2$ the ones that bound form a dense subset of the Teichm\"uller space. Thus, there are infinitely many of them in each genus $g \geq 2$. However, there are only finitely many arithmetic ones, by \cite{BGLS}. The argument of \autoref{thm1} applies in this case, and we obtain

\begin{teo}\label{thm3}
Let $\alpha(g) = $ the number of non-isometric orientable closed arithmetic surfaces of genus $\leq g$ that bound geometrically. Then $(c g)^{\frac{g}{8}} \leq \alpha(g) \leq (d g)^{2 g}$, for some constants $0 < c \leq d$. 
\end{teo}

\begin{rem}\label{rem:surfaces-nc}
An analogous statement holds for finite-area non-compact surfaces if we substitute the genus $g$ with the area $x$. Namely, then $(c x)^{\frac{x}{32 \pi}} \leq \alpha(x) \leq (d x)^{\frac{x}{2 \pi}}$, for some $0 < c \leq d$.
\end{rem}

This adds many more (albeit not very explicit) examples to the ones obtained by Zimmermann in \cite{Z1, Z2}.

The manifolds that we construct in abundance in order to prove \autoref{thm1} -- \autoref{thm3} all happen to be orientation double covers. An easy observation implies that any closed orientable manifold $M$ that is an orientation cover bounds topologically: consider $W^\prime = M \times [0,1]$ and quotient one of its boundary components by an orientation-reversing fixed point free involution that $M$ necessarily has in this case. The resulting manifold $W$ is orientable  with boundary $\partial W \cong M$. Indeed, these are the manifolds that are \textit{not} orientation covers that may make the cobordism group non-trivial. 

Concerning geometrically bounding manifolds, we are not aware at the moment of any that does bound geometrically and that is not an orientation cover, in both compact and finite-volume cases.  

\section{Constructing geodesic boundaries by colourings}\label{sec: color}

\subsection{The right-angled dodecahedron} Let $\mathcal{D} \subset \mathbb{H}^3$ be a right-angled dodecahedron. By Andreev's theorem \cite{Andreev}, it is realisable as a regular compact hyperbolic polyhedron. Suppose that the faces of $\mathcal{D}$ are labelled with the numbers $1$, $\dots$, $12$ as shown in \autoref{fig:dodecahedron}. Let $s_i$ be the reflection in the supporting hyperplane of the $i$-th facet of $\mathcal{D}$, for $ i = 1, \dots, 12$, and let 
$\Gamma_{12} = \mathrm{Ref}(\mathcal{D}) = \langle s_1, s_2, \dots, s_{12} \rangle$ be the corresponding reflection group.

\begin{figure}[ht]
 \begin{center}
  \includegraphics[width = 6 cm]{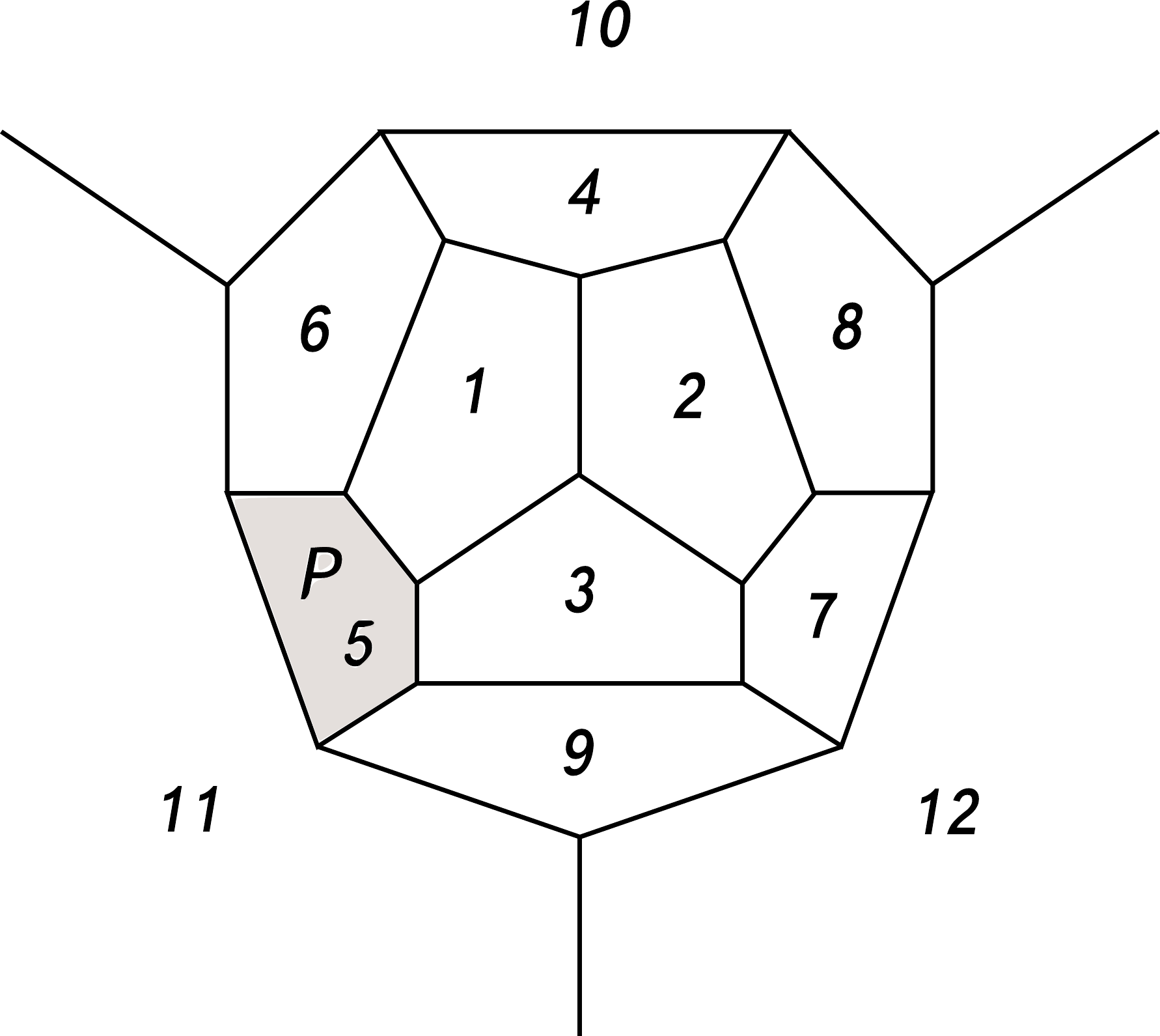}
 \end{center}
 \caption{A face labelling for the dodecahedron $\mathcal{D}$.}
 \label{fig:dodecahedron}
\end{figure}

Let $P$ be the pentagonal two-dimensional face of $\mathcal{D}$ labelled $5$ and let $\Gamma_4 =  \langle s_1, s_3, s_9, s_{11} \rangle$ be an infinite-index subgroup of $\Gamma_{12}$, which we may consider as a reflection group acting on the supporting hyperplane of $P$, which is isometric to $\mathbb{H}^2$. There is a retraction $R$ of $\Gamma_{12}$ onto $\Gamma_4$ given by
\begin{equation*}
R: s_i \mapsto  
\left\{
\begin{array}{cl}
s_i, &\mbox{ if } i \in \{1, 3, 9, 11\},\\
\mathrm{id}, &\mbox{ otherwise. }
\end{array} 
\right.
\end{equation*}

The group $\Gamma_4$ is virtually free: it contains $F_3 \cong \langle x, y, z \rangle$, a free group of rank $3$, as an index $8$ normal subgroup. Indeed, with $x = s_1 s_{11}$, $y = (s_1 s_9)^2$, $z = s_1 s_3 s_{11} s_3$, we have $F_3$ realised as a subgroup of $\Gamma_4$, which is the fundamental group of a $2$-sphere with four disjoint closed discs removed, as depicted in \autoref{fig:holed-sphere}. 

Let $P$ be a simple $n$-dimensional polytope (not necessarily hyperbolic) with $m$ facets labelled by distinct elements of $\Omega = \{1, 2, \dots, m\}$. A colouring of $P$, according to \cite{DJ, GS, I, Vesnin87, Vesnin}, is a map $\lambda: \Omega \rightarrow \mathbb{Z}^n_2$. A colouring is called proper if the colours of facets around each vertex of $P$ are linearly independent vectors of $V = \mathbb{Z}^n_2$. 

Proper colourings of compact right-angled polytopes $P\subset \mathbb{H}^n$ give rise to interesting families of hyperbolic manifolds \cite{GS, KMT, Vesnin87, Vesnin}. Such polytopes $P$ are necessarily simple.

\begin{figure}[hb]
\begin{center}
\includegraphics[width = 6 cm]{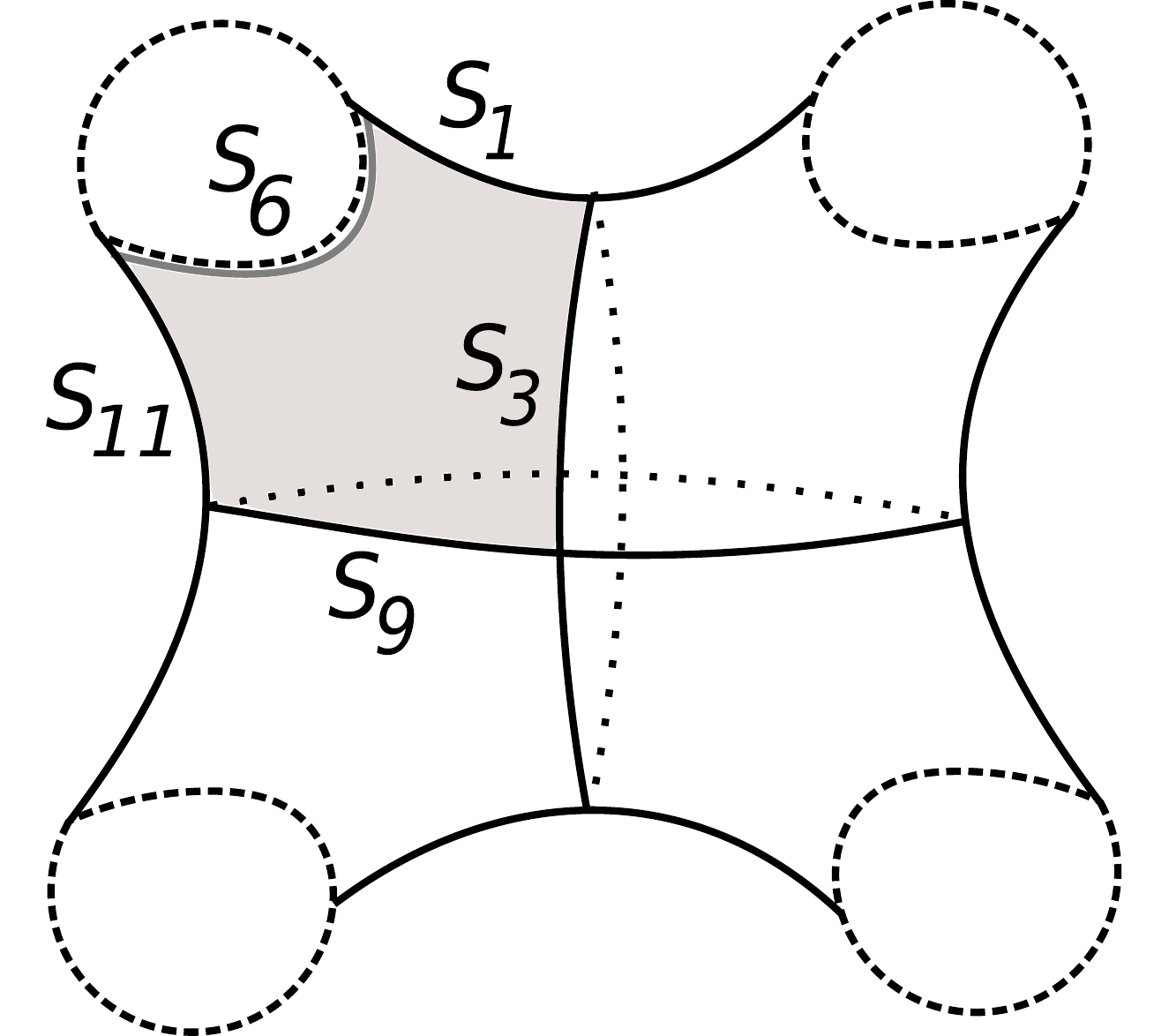}
\end{center}
\caption{The sphere $\mathbb{S}^2$ with four disjoint closed discs removed, and one of eight tiles associated to $\Gamma_4$ shaded. The reflection side of $s_6$ completes this tile to the pentagon $P$, and removing it is equivalent to cutting out a closed disc.}
 \label{fig:holed-sphere}
\end{figure}

In \cite{KMT} the notion of a colouring is extended to let $V = \mathbb{Z}^s_2$, $s \geq 2$, be a finite-dimensional vector space over $\mathbb{Z}_2$, and in \cite{KS} the notion of colouring is extended to polytopes that are not necessarily simple, but rather satisfy a milder constraint of being simple at edges. 

A polytope $P \subset \mathbb{H}^n$ is called simple at edges if each edge belongs to exactly $(n-1)$ facets. In the case of a finite-volume right-angled polytope $P \subset \mathbb{H}^n$,  $P$ is simple if $P$ is compact, and $P$ is simple at edges if it has any ideal vertices.

A colouring of a polytope $P\subset \mathbb{H}^n$ which is simple at edges is a map $\lambda: \Omega \rightarrow V$, where $V = \mathbb{Z}^s_2$, $s\geq n$, is a finite-dimensional vector space over $\mathbb{Z}_2$. A colouring $\lambda$ is proper if the following two conditions are satisfied:
\begin{enumerate}
\item \emph{Properness at vertices:} if $v$ is a simple vertex of $P$, then the $n$ colours of facets around it are linearly independent vectors of $V$;
\item \emph{Properness at edges:} if $e$ is an edge of $P$, then the $(n-1)$ colours of facets around $e$ are linearly independent. 
\end{enumerate}

Given a fixed labelling $\Omega$ of the facets of a finite-volume right-angled polytope $P \subset \mathbb{H}^n$, we shall write its colouring as a vector $\mathbf{\lambda} = (\lambda_1, \dots, \lambda_m)$, where $\lambda_i = \sum^{\dim V - 1}_{k = 0} \lambda(i)_k\cdot 2^k$ is a binary representation of the vector $\lambda(i) \in V$, for all $i \in \Omega$. 

Let $s_i$ be a reflection in the supporting hyperplane of the $i$-th facets of $P$. Then a proper colouring $\lambda: \Omega \rightarrow V$ defines a homomorphism from the reflection group $\Gamma = \mathrm{Ref}(P) = \langle s_1, s_2, \dots, s_m \rangle$ of $P$ to $V$, such that $\ker \lambda$ is a torsion-free subgroup of $\Gamma$ \cite{KS}.

Let us consider one of the colourings of $\mathcal{D}$ defined in \cite[Table 1]{GS}, that gives rise to a non-orientable manifold cover of the orbifold $\mathbb{H}^3 \diagup \Gamma_{12}$. Namely, choose $\mathbf{\lambda} = (1, 2, 4, 4, 2, 6, 3, 5, 5, 3, 1, 7)$, so that the $i$-th component of $\mathbf{\lambda}$ corresponds to the colour $\lambda_i$ of the $i$-th face of $\mathcal{D}$. As follows from \cite[Corollary 2.5]{KMT}, this colouring is indeed non-orientable, since $\lambda_1 + \lambda_2 + \lambda_7 = \mathbf{0}$ in $\mathbb{Z}^3_2$. Thus, $M = \mathbb{H}^3 \diagup \Gamma$, with $\Gamma  = \ker \mathbf{\lambda}$ a torsion-free subgroup of $\Gamma_{12}$, is a non-orientable compact hyperbolic $3$-manifold. 

The reflection group $\Gamma_{12}$ is an index $120$ subgroup in the reflection group $\mathrm{Ref}(T)$ of the orthoscheme $T = [4,3,5]$, which is arithmetic. Thus, $\Gamma_{12}$ is also arithmetic. Moreover, $\mathrm{Ref}(T) = O^+(q, \mathbb{Z}[\omega])$, with $\omega = \frac{1 + \sqrt{5}}{2}$,  for the quadratic form $q = -\omega x^2_0 + x^2_1 + x^2_2 + x^2_3$, as described in \cite[\S 7]{B} and, initially, in \cite{Bugaenko}.

Next, let $\rho: \Gamma \rightarrow R(\Gamma)$ be the restriction of $R$. Observe that $R(\Gamma) = \Gamma_4$, and thus $\rho: \Gamma \rightarrow \Gamma_4$ is an epimorphism. Here we use the fact that $s_1 = \rho(s_1 s_2 s_7)$, $s_3 = \rho(s_3 s_4)$, $s_9 = \rho(s_8 s_9)$, and $s_{11} = \rho(s_2 s_{10} s_{11})$, where all the respective products of $s_i$'s belong to $\Gamma  = \ker \mathbf{\lambda}$. 

For any subgroup $K \leq F_3$ of index $n$, let us consider $\rho^{-1}(K) = R^{-1}(K) \cap \Gamma$. Then $K$ has index $8n$ in $\Gamma_4$, and $H = \rho^{-1}(K)$ has index $8n$ in $\Gamma$.

Moreover, we produce an orientation-reversing element $\delta \in \Gamma$, such that $\delta \in H$ for every such $H$. 

Having established these facts, we know that there are $\asymp n^n$ non-conjugate in $\mathrm{Isom}(\mathbb{H}^3)$ subgroups of $\Gamma$ by using the argument of \cite[\S 5.2]{BGLS}, and thus there are $\asymp x^x$ non-isometric non-orientable compact arithmetic $3$-manifolds $M = \mathbb{H}^3 \diagup H$ of volume $\leq x$ (for $x > 0$ big enough). This proves the three-dimensional case of \autoref{prop1}.

Now, observe that $x = s_1 s_{11}$, and $\lambda(x) = (1,0,0)^t + (1,0,0)^t = \mathbf{0}$ in $\mathbb{Z}^3_2$.  Similarly, $\lambda(y) = \lambda(z) = \mathbf{0}$. Also, $R$ maps $x$, $y$, and $z$ respectively to themselves. Thus, $F_3 = \langle x, y, z \rangle \subset \rho(\Gamma)$. Finally, the element $\delta = s_2 s_4 s_6$ is such that $\lambda(\delta) = (0,1,0)^t + (1,0,0)^t + (1,1,0)^t = \mathbf{0}$, and $\rho(\delta) = \mathrm{id}$, so that $\delta \in \Gamma$ and $\delta \in H = \rho^{-1}(K)$, for every $K \leq  F_3$. 

Given that $H \leq O^+(q, \mathbb{Z}[\omega])$ for an admissible quadratic form $q$, we have that the argument in the proof of \cite[Corollary 1.5]{KRS} applies in this case, and thus the non-orientable compact manifold $M = \mathbb{H}^3\diagup H$ embeds into a compact orientable manifold $N = \mathbb{H}^4\diagup G$, for some arithmetic torsion-free $G \leq O^+(Q, \mathbb{Z}[\tau])$, with $Q = q + x^2_4$. Then, cutting $N$ along $M$ produces a manifold $N // M$, which is connected since $N$ is orientable while $M$ is not. Also, since $M$ is a one-sided submanifold of $N$, the boundary $\partial N'$ is isometric to $\widetilde{M}$, the orientation cover of $M$. Thus we obtain a collection of $\asymp n^n$ orientable arithmetic $3$-manifolds $\widetilde{M}$ that bound geometrically. However, some of them can be isometric, since the same manifold $\widetilde{M}$ can be the orientation cover of several distinct non-orientable manifolds $N_1$, $\dots$, $N_m$. 

In order to estimate $m$, observe that each $N_i$ is a quotient of $\widetilde{M}$ by a fixed point free orientation-reversing involution. Let the number of such involutions for $\widetilde{M}$ be $I(\widetilde{M})$. Then $m \leq I(\widetilde{M}) \leq | \mathrm{Isom}(\widetilde{M}) | \leq c_1 \cdot \mathrm{Vol}(\widetilde{M}) \leq c_2 \cdot n = c_3 x$. Indeed, the isometry group of $\widetilde{M}$ is finite, and by the Kazhdan-Margulis theorem \cite{KM} there exists a lower bound for the volume of the orbifold $\widetilde{M}\diagup \mathrm{Isom}(\widetilde{M}) \geq c_0 > 0$, from which the final estimate follows. Thus, we have at least $\asymp n^n/(c_2 n) \asymp n^n \asymp x^x$ non-isometric compact orientable arithmetic hyperbolic $3$-manifolds $\widetilde{M}$ of volume $\leq x$ that bound geometrically. The upper-bound of the same order of growth follows from \cite{BGLS}. This proves the three-dimensional case of \autoref{thm1}.

\subsection{The right-angled 120-cell} Let $\mathcal{C} \subset \mathbb{H}^4$ be the regular right-angled $120$-cell. This polytope can be obtained by the Wythoff construction with the orthoscheme $[4,3,3,5]$ that uses the vertex stabiliser subgroup $[3,3,5]$ of order $(120)^2 = 14400$. The polytope $\mathcal{C}$ is compact and each of its $3$-dimensional facets is a regular right-angled dodecahedron isometric to $\mathcal{D}$ defined above. 

Let us choose a facet $F$ of $\mathcal{C}$ and label it $120$. Since $F$ is isometric to $\mathcal{D}$, we can label the neighbouring facets of $F$ as follows:
\begin{itemize}
\item choose an isometry $\varphi$ between $F$ and $\mathcal{D}$ and transfer the labelling of $2$-dimensional faces of $\mathcal{D}$ depicted in \autoref{fig:dodecahedron} from $\mathcal{D}$ to $F$ via $\varphi$,
\item if $F'$, a facet of $\mathcal{C}$, shares a $2$-face labelled $i \in \{ 1, 2, \dots, 12\}$ with $F$, label $F'$ with $i$.
\end{itemize}
The remaining facets of $\mathcal{C}$ can be labelled with the numbers in $\{13, \dots, 119\}$ in an arbitrary way. Let $s_i$ denote the reflection on the supporting hyperplane of the $i$-th facet of $\mathcal{C}$, and let $\Gamma_{120} = \mathrm{Ref}(\mathcal{C}) = \langle s_1, s_2, \dots, s_{120} \rangle$.

Now define a colouring $\Lambda$ of $\mathcal{C}$ by using the colouring $\mathbf{\lambda}$ of $\mathcal{D}$ defined above. Namely, we set 
\begin{equation*}
\Lambda(s_i) = \left\{
\begin{array}{cl}
\lambda_i, &\mbox{ for } 1\leq i \leq 12,\\
2^{i-10}, &\mbox{ for } 13\leq i \leq 120.
\end{array}
\right.
\end{equation*}

Observe that $\Lambda$ is a proper colouring of $\mathcal{C}$, as defined in \cite{KMT}, and thus $\Gamma = \ker \Lambda$ is torsion-free. Also, $\Lambda$ is a non-orientable colouring. As in the case of $\mathcal{D}$, we use the retraction $R$ in order to map $\Gamma_{120}$ onto $\Gamma_4$, that contains $F_3$ as a finite-index subgroup. By taking preimages $H = \rho^{-1}(K)$ in $\Gamma$ of index $n$ subgroups $K\leq F_3$ and applying our argument from the previous section, we complete the proof of \autoref{prop1} in the $4$-dimensional case and obtain $\asymp n^n$ non-isometric non-orientable compact arithmetic hyperbolic $4$-manifolds $M = \mathbb{H}^4\diagup H$. The rest of the argument follows from \cite[Theorem 1.4]{KRS}. Thus, the $4$-dimensional case of \autoref{thm1} is also proven.

\subsection{Non-compact right-angled polytopes} Let $\mathcal{R}_3$ be a right-angled bi-pyramid depicted in \autoref{fig:bipyramid}, which is the first polytope in the series described by L. Potyaga\u{\i}lo and \`{E}. Vinberg in \cite{PV}. The construction in \cite{PV} produces a series of polytopes $\mathcal{R}_n \subset \mathbb{H}^n$, for $n=3, \dots, 8$, of finite volume, with both finite and ideal vertices, such that each facet of $\mathcal{R}_n$ is isometric to $\mathcal{R}_{n-1}$. Each $\mathcal{R}_n$ is produced by Whythoff's construction from the quotient of $\mathbb{H}^n$ by the reflective part of $O^+(f_n, \mathbb{Z})$, with $f_n = -x^2_0 + \sum^n_{k=1} x^2_k$, for $n=3, \dots, 8$. 

\begin{figure}
\begin{center}
\includegraphics[width = 6 cm]{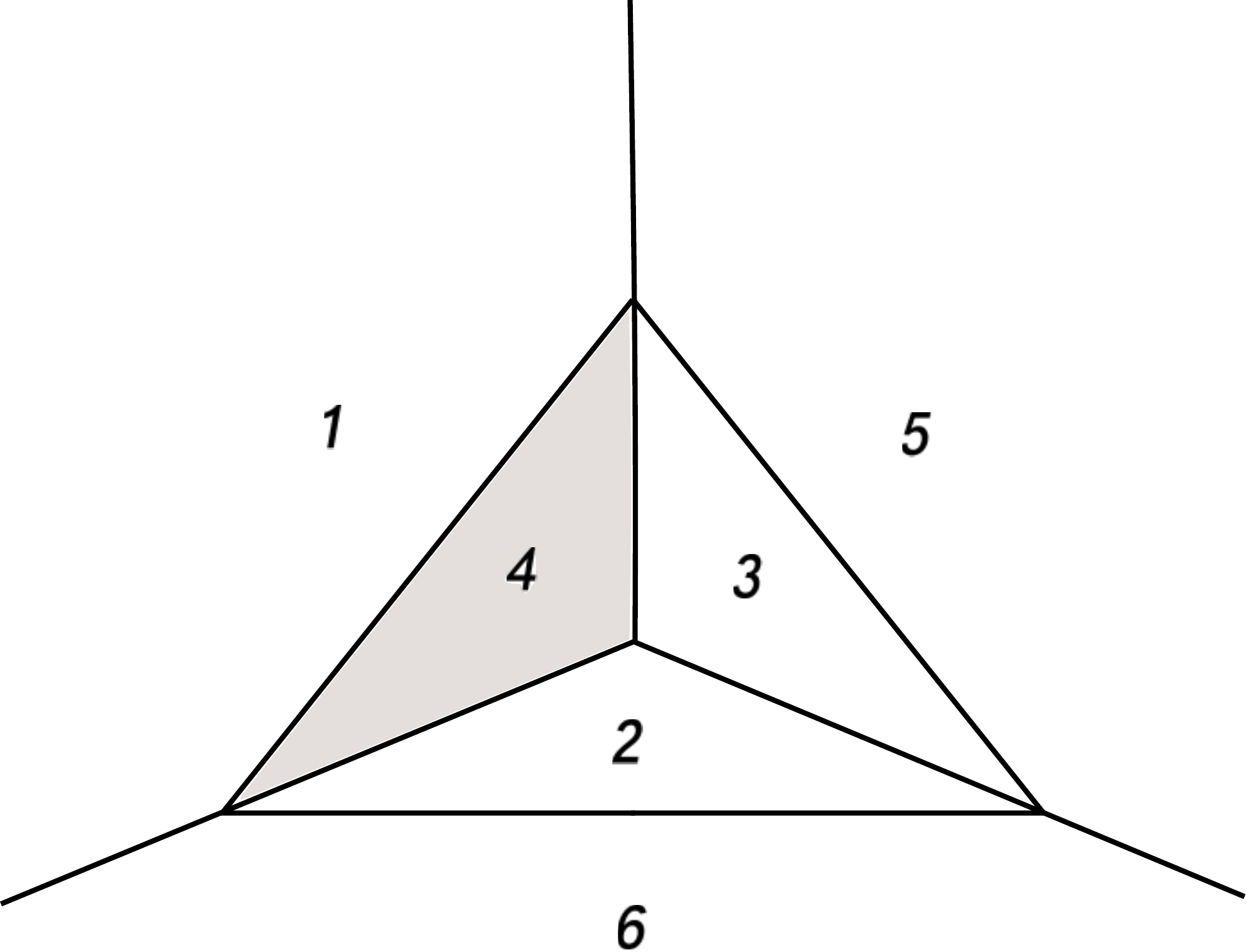}
\end{center}
\caption{A face labelling for the bi-pyramid $\mathcal{R}_3$. The compact vertices in $\mathbb{H}^3$ are the central one and the one at $\infty$. All other vertices are ideal and belong to $\partial \mathbb{H}^3$.}
\label{fig:bipyramid}
\end{figure}

If we provide a non-orientable proper colouring of $\mathcal{R}_3$, as defined in \cite{KS}, we can apply our previous reasoning in order to prove \autoref{prop2} and \autoref{thm2}, as consequence. Let us label the faces of $\mathcal{R}_3$ as shown in \autoref{fig:bipyramid}, and let the colouring be $\mathbf{\lambda} = (1, 1, 4, 7, 5, 2)$. It is easy to check that $\mathbf{\lambda}$ is indeed proper, since we need to check only the colours around the finite vertices and edges of $\mathcal{R}_3$. Also, $\mathbf{\lambda}$ is non-orientable, since $\lambda_4 + \lambda_5  + \lambda_6 = \mathbf{0}$ in $\mathbb{Z}^3_2$. Let $\Gamma = \ker \mathbf{\lambda}$.

Let $s_i$ be the reflection in the $i$-th facet of $\mathcal{R}_3$, and $\Gamma_6 = \langle s_1, \dots, s_6 \rangle$, and $\Delta = \langle s_1, s_2, s_3 \rangle$. Observe that $\Delta$ contains a free group of rank $2$ as a normal subgroup of index $4$. Indeed, $F_2 = \langle x, y \rangle$, with $x = s_1 s_2$, $y = s_3 s_1 s_2 s_3$ is such a subgroup. 

Let $R$ be a retraction $\Gamma_6 \rightarrow \Delta$ given by 
\begin{equation*}
R: s_i \mapsto  
\left\{
\begin{array}{cl}
s_i, &\mbox{ if } i \in \{1, 2, 3\},\\
\mathrm{id}, &\mbox{ otherwise. }
\end{array} 
\right.
\end{equation*}

Since $R$ maps $x$ and $y$ respectively to themselves, and $\lambda(s_1 s_2) = (1,0,0)^t + (1,0,0)^t = \mathbf{0}$ in $\mathbb{Z}^3_2$, we have that $F_3 \subset R(\Gamma)$. Moreover, for $\delta = s_4 s_5 s_6$ it holds that $\lambda(\delta) = \mathbf{0}$, as already verified above, and $R(\delta) = \mathrm{id}$. Then the argument from the previous case of the right-angled dodecahedron applies verbatim. 

For the induction step from $\mathcal{R}_{n-1}$ to $\mathcal{R}_n$ we just need to enhance the colouring in the way completely analogous to the extension of a non-orientable colouring of the dodecahedron $\mathcal{D}$ to a non-orientable colouring of the $120$-cell $\mathcal{C}$. Again, the rest of the argument proceeds verbatim in complete analogy to the previous cases.

\subsection{Surfaces that bound geometrically}
Let $\mathcal{P} \subset \mathbb{H}^2$ be a compact regular right-angled octagon, with sides labelled anti-clockwise $1$, $6$, $2$, $7$, $3$, $8$, $4$, $5$. Let $s_i$ be the reflection in the $i$-th side of $\mathcal{P}$, and $\Gamma_{8} = \mathrm{Ref}(\mathcal{P}) = \langle s_1, s_2, \dots, s_{8} \rangle$ be its reflection group. Also, let $\Gamma_4 = \langle s_1, s_2, s_3, s_4 \rangle$ and $F_3 = \langle x, y, z \rangle$, with $x = s_1 s_2$, $y = s_1 s_3$, $z = s_1 s_4$, be a free subgroup of $\Gamma_4$ of index $2$. The retraction of $\Gamma_8$ onto $\Gamma_4$ is given by
\begin{equation*}
R: s_i \mapsto  
\left\{
\begin{array}{cl}
s_i, &\mbox{ if } i \in \{1, 2, 3, 4\},\\
\mathrm{id}, &\mbox{ otherwise. }
\end{array} 
\right.
\end{equation*}

Let us choose a colouring $\mathbf{\lambda} = (1, 1, 1, 1, 2, 3, 5, 6)$ for $\mathcal{P}$, which is a proper and non-orientable one, since $\lambda_1 + \lambda_5 + \lambda_6 = \mathbf{0} \in \mathbb{Z}^3_2$. Let $\Gamma = \ker \lambda$. An easy check ensures that $F_3 \subset R(\Gamma)$, as well as that $R(\delta) = \mathrm{id}$ for an orientation-reversing element $\delta = s_6 s_7 s_8 \in \Gamma$. Then the lower bound $\alpha(x) \geq (c x)^{\frac{x}{32 \pi}}$, for some constant $c > 0$, for the number of geometrically bounding surfaces of area $\leq x$ (for $x$ large enough) follows immediately: the area of $\mathcal{P}$ equals $2 \pi$, $\Gamma$ has index $8$ in $\Gamma_8$, and the orientation cover of a non-orientable surface has twice its area. We also use the fact that the rank $d\geq 2$ free group $F_d$ has $\geq (n!)^{d-1}$ subgroups of index $\leq n$, for $n$ large enough. The upper bound  $\alpha(x) \leq (d x)^{\frac{x}{2\pi}}$, for some constant $d \geq c > 0$, follows from \cite{BGLS}. Since $\text{area} = 4 \pi (g-1)$, for an orientable genus $g\geq 2$ surface, this proves \autoref{thm3}. The case of non-compact finite-area surfaces mentioned in Remark~\ref{rem:surfaces-nc} proceeds by analogy. 

\section{Constructing geodesic boundaries by arithmetic reductions}\label{sec: reductions}
We start by recalling the following lemma of Long and Reid \cite[Lemma~2.2]{LR3} (c.f. also the remark after its proof).

\begin{lemma}[Subgroup Lemma]\label{lem:TF-subgroup}
Let $\Gamma < O^+(n, 1)$ be a subgroup of hyperbolic isometries defined over a number field $K$, and $\delta$ an element of $\Gamma$. Let $\theta_1, \theta_2: \Gamma \rightarrow F_i$ be two homomorphisms of $\Gamma$ onto a group $F_i$, with torsion-free kernels. Let $\Theta(g) = \left( \theta_1(g), \theta_2(g) \right): \Gamma \rightarrow F_1 \times F_2$. Suppose that $\theta_i(\delta)$ has order $k_i < \infty$, $i = 1, 2$, and any prime dividing $gcd(k_1, k_2)$ appears with distinct exponents in $k_1$ and $k_2$. Then $\Theta^{-1} \langle  (\theta_1(\delta), \theta_2(\delta)) \rangle$ is a torsion-free subgroup in $\Gamma$ of finite index that contains $\delta$.
\end{lemma}

The following lemma is used in order to show that the maps that we choose in the sequel as $\theta_i$, $i = 1, 2$, in the Subgroup Lemma above have torsion-free kernels. Its proof is very similar to that of \cite[Lemma~2.4]{LR3}. 

\begin{lemma}[No Torsion Lemma]\label{lem:No-torsion}
Let $\Gamma < O^+(n, 1)$ be a finite subgroup defined over the ring of integers $\mathcal{O}_K$ of a number field $K$, and let $p \in \mathcal{O}_K$ be an odd rational prime that does not divide the order of $\Gamma$. Then the reduction of $\Gamma$ modulo the ideal $\mathcal{J} = (p)$ is isomorphic to $\Gamma$.  
\end{lemma}
\begin{proof}
A non-trivial element $g$ of the kernel of the reduction map $\Gamma(\mathcal{O}_K) \rightarrow \Gamma(\mathcal{O}_K/\mathcal{J})$ can be written in the form $g = \mathrm{id} + p^r h$, where $h$ is a matrix not all of whose entries are divisible by $p^r$, with $r$ some positive integer. Let $q < \infty$ be the order of an element $g \in \Gamma$. Then we get
\begin{equation*}
\mathrm{id} = g^q = \mathrm{id} + q p^r h + \sum^q_{t=2} \binom{q}{t} p^{rt} h^{t}, 
\end{equation*}
and thus 
\begin{equation*}
q h = \mathbf{0} \mod p^r.
\end{equation*}
The latter implies $p^r$ divides $q$, since $h \neq \mathbf{0} \mod p^r$. Thus, $p$ divides $q$, and $q$ divides the order of $\Gamma$, since $q$ is the order of an element of $\Gamma$. The latter is a contradiction, and thus the reduction map has trivial kernel. 
\end{proof}

As shown by Vinberg in \cite{Vinberg}, in some cases for an admissible quadratic form $q$ of signature $(n, 1)$ defined over a totally real number field $K$ with ring of integers $O_K$ it holds that $O^+(q, O_K) = \mathrm{Ref}(P) \rtimes \mathrm{Sym}(P)$, where $P \subset \mathbb{H}^n$ is a finite-volume polytope. Here, $\mathrm{Ref}(P)$ denotes the associated reflection group, and  $\mathrm{Sym}(P)$ is the group of symmetries of $P$. Also, we assume that $O_K$ is a principal ideal domain in order to keep our account simpler. We refer the reader to \cite{Guglielmetti-1} for more details.

If the above presentation of $O^+(q, O_K)$ takes place for some finite-volume polytope $P \subset \mathbb{H}^n$, the form $q$ is called \textit{reflective}, and the polytope $P$ is called its \textit{associated polytope}.

An algorithm introduced by Vinberg in \cite{Vinberg} and implemented in \cite{Guglielmetti-2} by Guglielmetti allows us to find the associated polytope $P\subset\mathbb{H}^n$ in finite time, for any reflective admissible quadratic form of signature $(n, 1)$. 

\subsection{Compact polytopes in dimensions 5 and 6} \label{compact56}
Let $\omega=\frac{1+\sqrt{5}}{2}$ and let $P_n \subset \mathbb{H}^n$ be the polytopes associated to the quadratic forms 
\begin{equation*}
q_5=-(-1+2\omega)x_0^2+x_1^2+x_2^2+x_3^2+x_4^2+x_5^2,
\end{equation*}
\begin{equation*}
q_6=-2\omega x_0^2+x_1^2+x_2^2+x_3^2+x_4^2+x_5^2+x_6^2.
\end{equation*}
The polytopes $P_5$ and $P_6$ are apparently new, and were found by using AlVin and CoxIter software \cite{Guglielmetti-2, Guglielmetti-3}. They differ substantially from the polytopes that appear in \cite{Bugaenko, Bugaenko2}, and have fewer facets. 

Let $\Gamma_n = \mathrm{Ref}(P_n)$, be the reflection group of $P_n$, with generators $s_i$, $i \in I_n$,  where $I_n$ is the set of outer normals to the facets of $P_n$ or, equivalently, the set of nodes of the Coxeter diagram of $P_n$\footnote{The notation used is as follows:  a dashed edge means two reflection hyperplanes have a common perpendicular, a solid edge means parallel (at the ideal boundary) hyperplanes, a double edge means label $4$, a single edge means label $3$, any other edge has a label on it describing the corresponding dihedral angle. The colours are used for convenience only. }. With standard basis $\{v_0, v_1, \dots, v_n\}$, the Vinberg algorithm determines the outer normals for $n=5, 6$ which are given in the Appendix.

The associated reflection group $\Gamma_n$ is arithmetic, and contains a virtually free parabolic subgroup
\begin{equation*}
\Delta = \begin{cases} 
\langle s_5, s_6, s_9 \rangle, &\text{for } n = 5, \\
\langle s_6, s_9, s_{17}  \rangle, &\text{for } n = 6;
\end{cases}
\end{equation*}
Indeed, $\Delta$ is isomorphic to the $(2,\infty,\infty)$-triangle group\footnote{Here $\Delta$ is not actually generated by reflections in the sides of a hyperbolic finite-area triangle, but is rather only abstractly isomorphic to such a group. However, we are interested in its algebraic rather than geometric properties, regarding its subgroup growth.}, which contains $F_2$ as a subgroup of index $4$. 

The retraction $R: \Gamma_n \rightarrow \Delta$ is defined by sending all but three generators of $\Gamma_n$ to $\mathrm{id}$, with the only generators mapped identically being those of $\Delta < \Gamma_n$.

In order for $R$ being well-defined, we essentially need that the generators of $\Delta$ be connected to the rest of the diagram by edges with even labels only, since any two generators connected by a path of odd-labelled edges are conjugate. This folds, for instance, if the facets corresponding to the reflections generating $\Delta$ are \textit{redoubleable} in terms of \cite{Allcock}.

The element 
\begin{equation*}
\delta = \begin{cases}
s_1 s_2 s_3 s_4 s_7, & \text{for } n=5, \\
s_7 s_{13} s_{18}, & \text{for } n=6.
\end{cases}
\end{equation*}
is orientation-reversing, as it is a product of an odd number of reflections in $\mathbb{H}^n$. Moreover, $\delta \in \ker R$.

\begin{table}
\begin{tabular}{|c|c|c|}
\hline
Polytope  & Diagram & LCM \\ \hline
$P_5$  & \begin{tabular}[c]{@{}c@{}}\includegraphics*[scale=0.20]{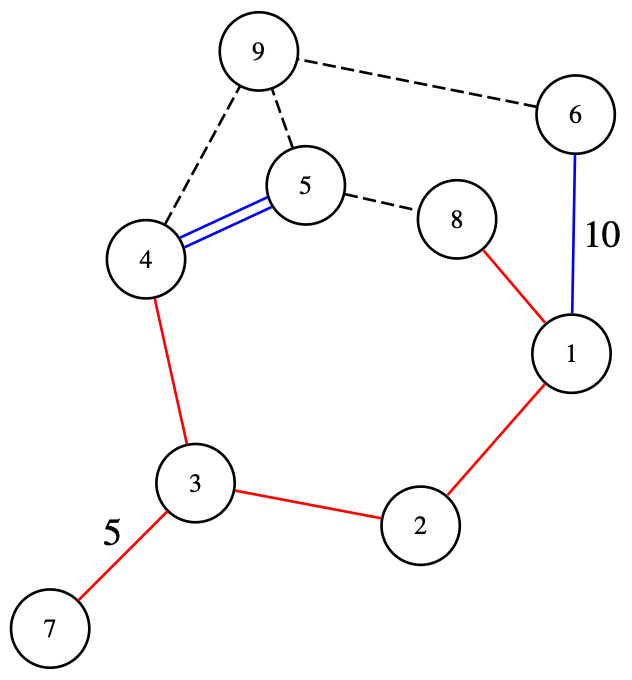}\end{tabular} & $57600 = 2^8 \cdot 3^2 \cdot 5^2$ \\ \hline
$P_6$  & \begin{tabular}[c]{@{}c@{}}\includegraphics*[scale=0.13]{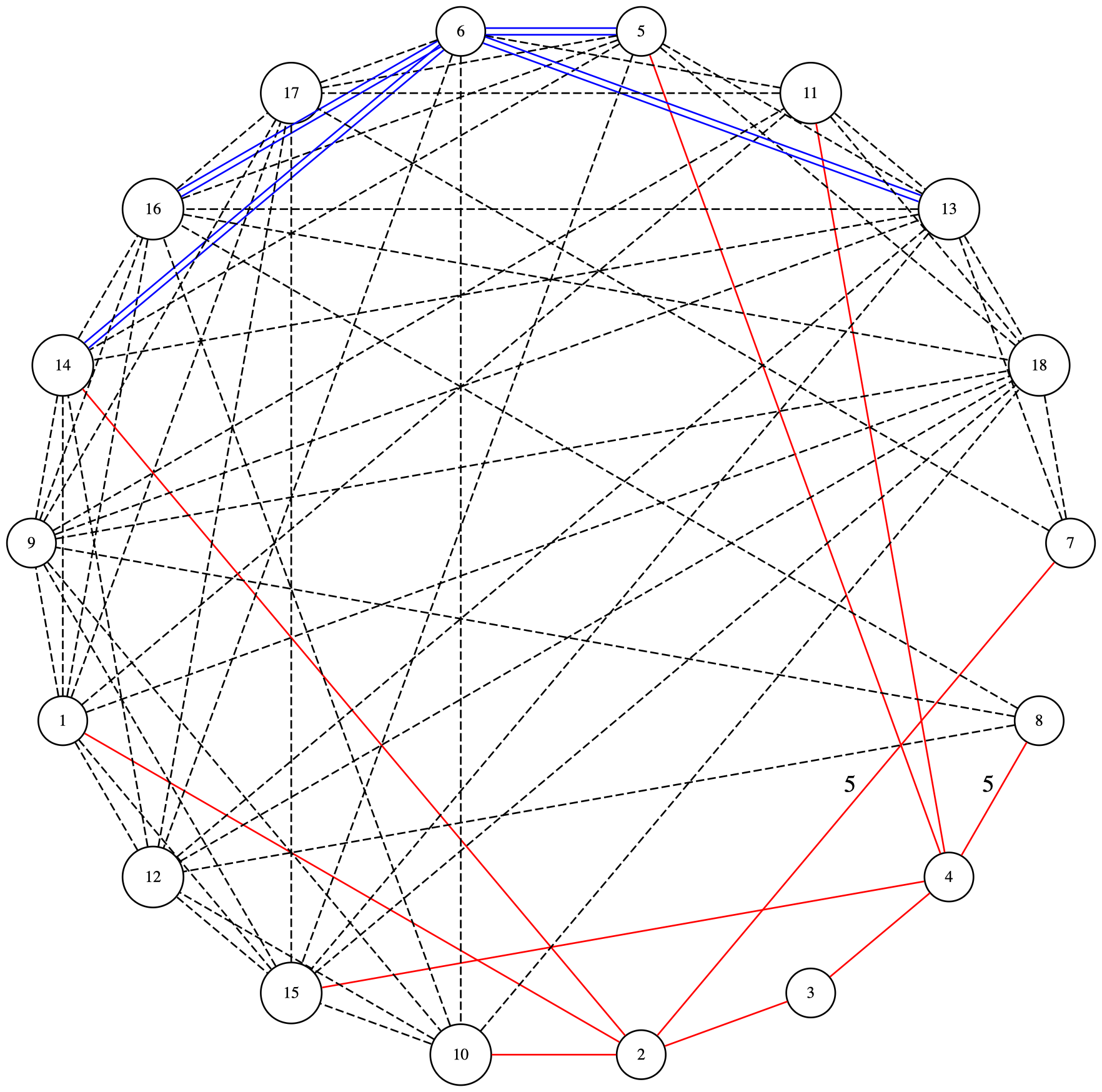}\end{tabular} & $230400 = 2^{10}\cdot 3^2\cdot 5^2$ \\ 
\hline
\end{tabular}
\caption{Polytopes $P_n$, $n=5, 6$, their Coxeter diagrams, and the least common multiple (LCM) of the orders of their parabolic finite subgroups}
\label{tab2}
\end{table}

We shall set the map $\Theta$ from \autoref{lem:TF-subgroup} to be a pair of reductions modulo various rational primes, and then use \autoref{lem:No-torsion} in order to ensure that their kernels are torsion-free. Indeed, we need to choose such an odd prime $p \in \mathbb{Z}$ that it does not divide the order of any finite parabolic subgroup in the diagram of $P_n$, $n = 5, 6$. The least common multiples of orders of finite parabolic subgroups for $P_n$, $n = 5, 6$, are given in~\autoref{tab2}\footnote{The orders of all finite parabolic subgroups associated with $P_n$ can be obtained by using \texttt{CoxIter} \cite{Guglielmetti-2} with the \texttt{-debug} option.}. 

For $\Gamma < \GL(n+1,\mathcal{O}_K)$, where $\mathcal{O}_K$ is the ring of integers of a number field $K$, let $\phi_p$ denote the homomorphism $\Gamma\rightarrow \GL(n+1,\mathcal{O}_K/\mathcal{J})$ induced by reduction modulo $\mathcal{J} = (p)$, the principal ideal generated by a rational integer $p$. 

Let us consider the reductions $\phi_{7}$ and $\phi_{11}$ as defined above, and let $\Theta = (\phi_{7}, \phi_{11})$. For $n=5$, the order of $\phi_{7}(\delta)$ equals $800 = 2^5 \cdot 5^2$, while the order of $\phi_{11}(\delta)$ equals $8052 = 2^2\cdot 3^1\cdot 11^1\cdot 61^1$, as follows by straightforward computations, c.f. \cite{CK-code1, CK-code2}.  For $n=6$, the order of $\phi_{7}(\delta)$ equals $8 = 2^3$, while the order of $\phi_{11}(\delta)$ equals $44 = 2^2\cdot 11^1$.

Then \autoref{lem:TF-subgroup} and \autoref{lem:No-torsion} apply, and $\Gamma = \Theta^{-1}\langle (\phi_{7}(\delta), \phi_{11}(\delta)) \rangle$ is a torsion-free subgroup of finite index in $\Gamma_n$ that contains the orientation-reversing element $\delta$ and retracts onto the free group $\Gamma\cap\Delta$. Then the argument analogous to that of Section~\ref{sec: color} applies. 

\subsection{Right-angled cusped polytopes in dimensions 4 to 8}\label{cusped45678}
Let $P_n \subset \mathbb{H}^n$ be the right-angled polytopes associated to the principal congruence subgroups of level $2$ for the quadratic forms 
\begin{equation*}
f_n = -x_0^2 + x_1^2 + x_2^2 + \cdots + x_n^2, \text{ for } n=4, \dots, 8.
\end{equation*}

Let $\Gamma_n = \mathrm{Ref}(P_n)$, be the associated reflection group, with generators $s_i$, $i \in I_n$,  where $I_n$ is the set of outer normals to the facets of $P_n$. With standard basis $\{v_0, v_1, \dots, v_n\}$, the Vinberg algorithm starts with the first $n$ outer normals being
\begin{equation*}
e_i = - v_i, \text{ for } 1 \leq i \leq n,
\end{equation*}
and continues with the next $\binom{n}{2}$ outer normals 
\begin{equation*}
e_{i, j} = v_0 + v_i + v_j, \text{ for } 1 \leq i < j \leq n,
\end{equation*}
all of them being $1$-roots, as is necessary for determining the reflective part of the principle congruence level $2$ subgroup, rather than that of the whole group of units for $f_n$. Let us set $e_{n+1} = e_{1, 2} = v_0 + v_1 + v_2$ and $e_{n+2} = e_{3, 4} = v_0 + v_3 + v_4$, as a more convenient notation.

Such $\Gamma_n$ is arithmetic, and it contains a virtually free parabolic subgroup $\Delta = \langle s_3, s_{4}, s_{n+2} \rangle$. Indeed, $\Delta$ is isomorphic to the $(2, \infty, \infty)$-triangle group which contains $F_2$ as a subgroup of index $4$. Consider the retraction
\begin{equation*}
R: s_i \mapsto  
\left\{
\begin{array}{cl}
s_i, &\mbox{ if } i \in \{3, 4, n+2\},\\
\mathrm{id}, &\mbox{ otherwise. }
\end{array} 
\right.
\end{equation*}

The element $\delta= s_1 s_2 s_{n+1}$ is orientation-reversing, as it is a product of $3$ reflections in $\mathbb{H}^n$. Moreover, $\delta \in \ker R$.

For $\Gamma\in \GL(n+1,\Z)$, let $\phi_m$ denote the homomorphism $\Gamma\rightarrow \GL(n+1,\Z/m\Z)$ induced by reduction modulo a positive integer $m$.  By \cite[Theorem~IX.7]{Newman} we know that the kernel of $\phi_m$ is torsion-free for $m > 2$. The reduction of $\delta$ modulo $3$ has order $4=2^2$, while its reductions modulo $4$ has order $2$, c.f. \cite{CK-code1, CK-code2}. Letting $\Theta = (\phi_3, \phi_4)$, \autoref{lem:TF-subgroup} applies, and $\Gamma = \Theta^{-1}\langle(\phi_3(\delta), \phi_4(\delta))\rangle$ is a torsion-free subgroup of finite index in $\Gamma_n$ that contains the orientation-reversing element $\delta$ and retracts onto the free group $\Gamma\cap\Delta$. 

\subsection{Cusped polytopes in dimensions 9 to 13}\label{cusped9to13}
Let $P_n \subset \mathbb{H}^n$ be the polytopes from Table~7 in \cite{Vinberg} associated to the quadratic forms 
\begin{equation*}
f_n = -2 x_0^2 + x_1^2 + x_2^2 + \cdots + x_n^2 \text{ for } n=9, 10 \text{ and } 13,
\end{equation*}
while $P_n$, for $n=11, 12$, be the polytopes with Coxeter diagrams given in Figure~\ref{fig:polytopes-cusped-11-12}, respectively. The latter ones appear to be new, and were found by using \texttt{AlVin} \cite{Guglielmetti-3}.

Let $\Gamma_n = \mathrm{Ref}(P_n)$, be the reflection group of $P_n$, with generators $s_i$, $i \in I_n$, where $I_n$ is the set of nodes in the Coxeter diagram of $P_n$. Such $\Gamma_n$ is arithmetic, and it contains a virtually free parabolic subgroup $\Delta$ indicated in Table~\ref{tab0}.

\begin{figure}[ht]
\hfill
\subfigure[$f_{11} = -2 x^2_0 + x^2_1 + \ldots + x^2_8 + 2 x^2_9 + 2 x^2_{10} + 2 x^2_{11}$]{\includegraphics[width=.47\textwidth]{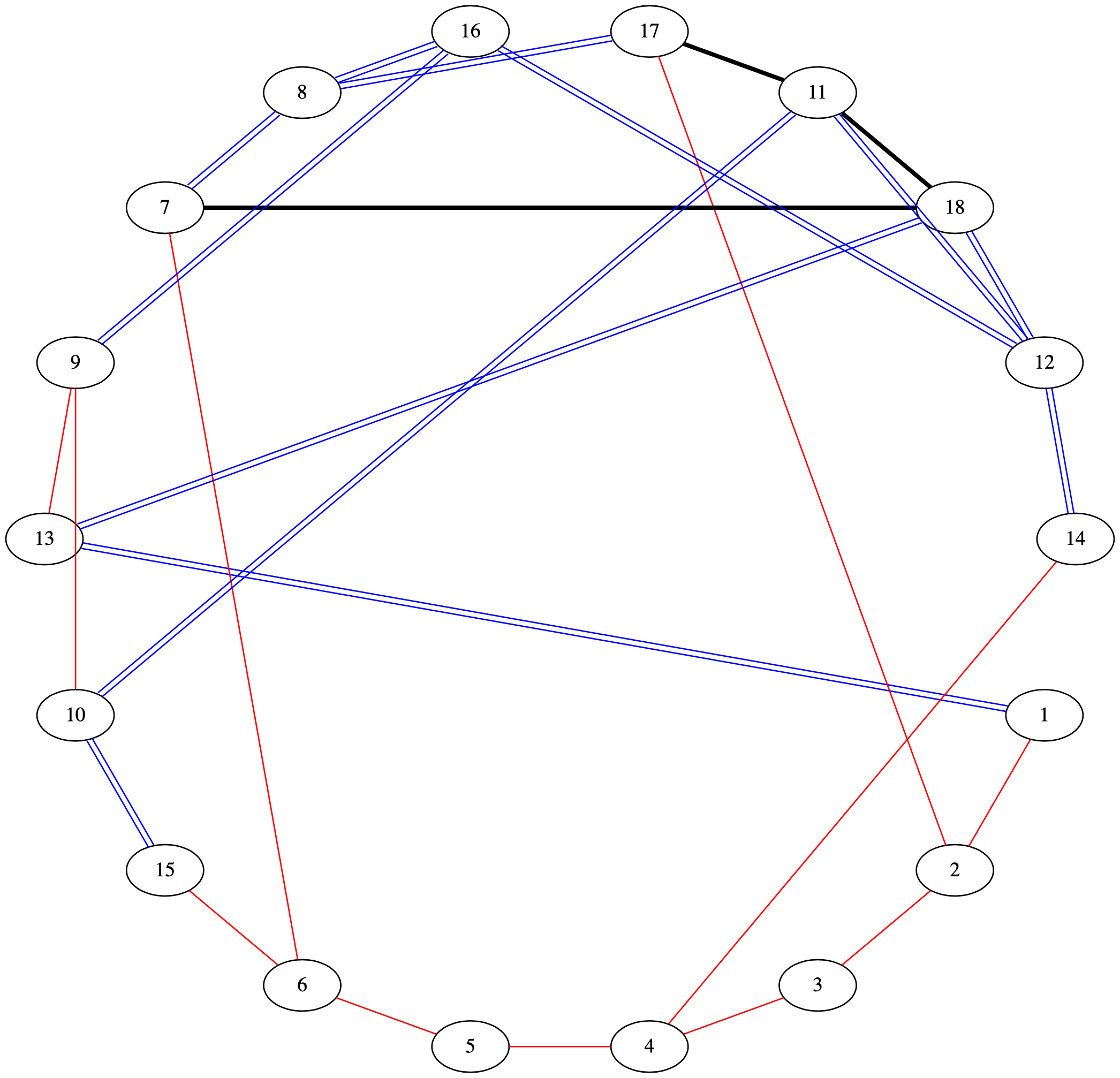}}
\hfill
\subfigure[$f_{12} = -2 x^2_0 + x^2_1 + \ldots + x^2_{11} + 2 x^2_{12}$]{\includegraphics[width=.47\textwidth]{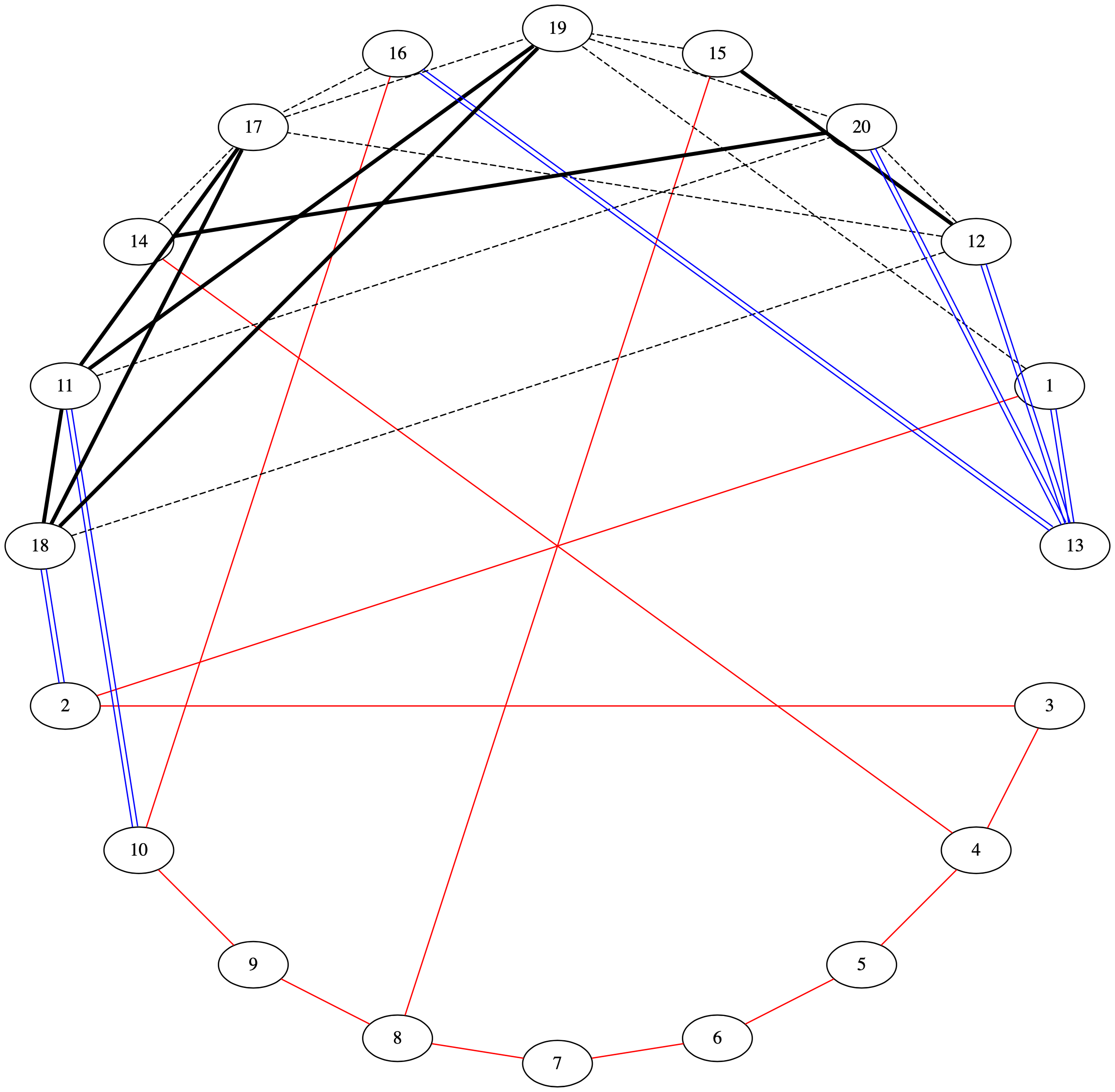}}
\hfill
\caption{The Coxeter diagrams for (a) $P_{11}$ and (b)  $P_{12}$, together with their associated quadratic forms}\label{fig:polytopes-cusped-11-12}
\end{figure}

Since $\Delta$ is generated by reflections in redoubleable facets we can define a retraction $R: \Gamma_n \rightarrow \Delta$, as before, that send all the generators of $\Gamma_n$ to $\mathrm{id}$, except of those of $\Delta$. 

\begin{table}[ht]
\begin{tabular}{|c|c|c|}
\hline
$n$  & generators of $\Delta$ & triangle group $\cong \Delta$ \\ \hline
$9$   &    $s_9$, $s_{10}$, $s_{12}$ & $(2, \infty, \infty)$ \\
$10$ & $s_{10}$, $s_{11}$, $s_{13}$ & $(2, 4, \infty)$ \\
$11$  & $s_{11}$, $s_{12}$, $s_{18}$ & $(4, 4, \infty)$ \\
$12$ & $s_{12}$, $s_{13}$, $s_{20}$ & $(4, 4, \infty)$ \\
$13$ & $s_{13}$, $s_{14}$, $s_{19}$ & $(2, \infty, \infty)$ \\
\hline
\end{tabular}
\caption{A virtually free parabolic subgroup $\Delta < \Gamma_n$, for $n = 9, \ldots, 13$}
\label{tab0}
\end{table}

The orientation-reversing element $\delta_n \in \ker R$ is defined by
\begin{equation*}
\delta_n = \begin{cases}
s_1 s_2 s_3 \ldots s_8 s_{n+2}, & \text{for } n = 9, 10, 13, \\
s_7 s_8 s_{16}, & \text{for } n = 11, \\
s_2 s_{11} s_{18}, & \text{for } n = 12.
\end{cases}
\end{equation*}

Letting $\Theta = (\phi_{m_1}, \phi_{m_2})$, \autoref{lem:TF-subgroup} applies with $m_1$ and $m_2$ as in \autoref{tab1}. Here, we also notice that $\delta_n$ for $n= 10, 13$ is an extension of $\delta_9$ by the identity map, which simplifies the computations. 

Then $\Gamma = \Theta^{-1}\langle(\phi_{m_1}(\delta), \phi_{m_2}(\delta))\rangle$ is a torsion-free subgroup of finite index in $\Gamma_n$ that contains the orientation-reversing element $\delta$ and retracts onto the free group $\Gamma\cap\Delta$. 

\medskip

\begin{table}[h]
\begin{tabular}{|c|c|c|c|c|}
\hline
$n$  & $m_1$ & $k_1 = \mbox{order of } \phi_{m_1}(\delta_n)$ & $m_2$ & $k_2 = \mbox{order of } \phi_{m_2}(\delta_n)$ \\ \hline
9, 10, 13  & $3$ & $84 = 2^2\cdot 3^1\cdot 7^1$  & $4$ & $34 = 2^1\cdot 17^1$ \\
11, 12 & $3$ &  $6 = 2^1 \cdot 3^1$ & $4$ & $4 = 2^2$ \\
\hline
\end{tabular}
\caption{Orders of the reductions of $\delta$ and their prime factorisations, c.f. \cite{CK-code1, CK-code2}}
\label{tab1}
\end{table}

\section{Constructing geometric boundaries from virtual retracts}

Let $q_n$ be an admissible quadratic form of signature $(n, 1)$ defined over a totally real number field $K$ with ring of integers $O_K$, and let $q_{n+1} = q_n + x^2_{n+1}$. Suppose also that $\Gamma_n < O^+(q_n, O_K)$ is a torsion-free subgroup, either of finite co-volume or co-compact.

Now assume that there exists a retraction $R_n: \Gamma_n \rightarrow \Delta$ of $\Gamma_n$ onto a virtually free subgroup $\Delta$ such that $\ker R_n$ contains an orientation-reversing element $\delta$.

By \cite[Proposition 7.1]{KRS}, there exists a torsion-free finite-index subgroup $\Gamma'_{n+1} \leq O^+(q_{n+1}, O_K)$, such that $\Gamma_n < \Gamma'_{n+1}$. Moreover, we may assume that  $\mathbb{H}^n / \Gamma_n$ is a properly embedded totally geodesic submanifold of $\mathbb{H}^{n+1} / \Gamma_{n+1}$. Thus, the group $\Gamma_n$ is a geometrically finite subgroup of $\Gamma'_{n+1}$.
By \cite[Theorem 1.4]{BHW}, there is a finite index subgroup $G<\Gamma'_{n+1}$, such that $G$ virtually retracts to its geometrically finite subgroups. In particular, $G$ virtually retracts to $G\cap\Gamma_n$. However, since $\Gamma'_{n+1}$ is linear, the arguments of \cite[Theorem 2.10]{LR4} apply to give a virtual retraction from $\Gamma'_{n+1}$ onto $\Gamma_n$. 
Let $\Gamma_{n+1}$ be the finite index subgroup of $\Gamma'_{n+1}$ which retracts onto $\Gamma_n$. Then the composition $R_{n+1}: \Gamma_{n+1} \rightarrow \Gamma_n \rightarrow \Delta$ is a retraction of $\Gamma_{n+1}$ onto $\Delta$, such that $\delta \in \ker R_{n+1}$.

All the previous arguments from Section \ref{sec: reductions} apply, and we obtain Theorems~\ref{thm2} -- \ref{thm3} for all $n \geq 2$, since we can use any of our examples worked out in Sections~\ref{sec: color}--\ref{sec: reductions} as a basis for the above inductive procedure. 

\section*{Acknowledgements}
\noindent
{{\small 
The first author thanks Darren Long for many helpful discussions. Both authors would like to thank Mikhail Belolipetsky (IMPA), Vincent Emery (Universit\"{a}t Bern), and Ruth Kellerhals (Universit\'e de Fribourg), as well as the Mathematisches Forschungsinstitut Oberwolfach (MFO) administration, for organising the mini-workshop ``Reflection Groups in Negative Curvature'' (1915b) in April 2019, during which the results of this paper were presented and discussed.
They also thank the organisers of the Borel Seminar (Les Diablerets, Switzerland) in December 2018, where this work was started, for stimulating research atmosphere, and the Swiss Doctoral Program -- CUSO, Swiss Mathematical Society, and ETH Z\"{u}rich for the financial support of the event. 
A word of special gratitude is addressed to the creators and maintainers of SAGE Math, and to Rafael Guglielmetti, the author of \texttt{CoxIter} and \texttt{AlVin}. The aforementioned software made most of the calculations in the present work possible and provided means for additional verification of those previously done.
We also thank the anonymous referees for their suggestions and criticism that helped us improving this paper.}}

\newpage

\appendix
\section{} \label{app}

\setlength{\abovedisplayskip}{0pt}
\setlength{\belowdisplayskip}{0pt}
\allowdisplaybreaks

\subsection{Outer normals for compact \texorpdfstring{$P_5$}{P5}} (\autoref{compact56})
\begin{flalign*}
&e_i = -v_i+v_{i+1} \text{ for } 1 \leq i\leq 4, &\\ 
&e_5 =-v_5, &\\ 
&e_6 =\omega v_0 + (2 + \omega) v_1, &\\
&e_7 =\omega (v_0 + v_1 + v_2 + v_3), &\\ 
&e_8 =(1 + \omega) (v_0 + v_1) +  \omega (v_2 + v_3 + v_4 + v_5). &
\end{flalign*}

\subsection{Outer normals for compact \texorpdfstring{$P_6$}{P6}} (\autoref{compact56})
\begin{flalign*}
&e_i = -v_i+v_{i+1} \text{ for } 1 \leq i\leq 5, &\\ 
&e_6 =-v_6, &\\ 
&e_7 = v_0 + w (v_1 + v_2), &\\
&e_8 = \omega (v_0 + v_1 + v_2 + v_3 + v_4), &\\
&e_9 = \omega v_0 + 2 \omega v_1, &\\
&e_{10} = (1 + \omega) (v_0 + v_1 + v_2) + \omega (v_3 + v_4 + v_5 + v_6), &\\
&e_{11} = (1 + 2 \omega) v_0 + (1 + 3 \omega) v_1 + (1 + \omega) (v_2 + v_3 + v_4) + \omega (v_5 + v_6), &\\
&e_{12} = (1 + 2 \omega) v_0 + (2 + 3 \omega) v_1 + \omega (v_2 + v_3 + v_4 + v_5 + v_6), &\\
&e_{13} = (2 + 2 \omega) v_0 + (1 + 2 \omega) (v_1 + v_2 + v_3 + v_4 + v_5) + v_6, &\\
&e_{14} = (2 + 3 \omega) v_0 + (2 + 4 \omega) v_1 + (2 + 2 \omega) v_2 + (1 + 2 \omega) (v_3 + v_4 + v_5) + v_6, &\\
&e_{15} = (2 + 3 \omega) v_0 + (3 + 4 \omega) v_1 + (1 + 2 \omega) (v_2 + v_3 + v_4) + 2 \omega v_5, &\\
&e_{16} = (2 + 4 \omega) v_0 + (3 + 6 \omega) v_1 + (1 + 2 \omega) (v_2 + v_3 + v_4 + v_5) + v_6, &\\
&e_{17} = (3 + 4 \omega) v_0 + (2 + 5 \omega) v_1 + (2 + 3 \omega) (v_2 + v_3 + v_4 + v_5) + \omega v_6, &\\
&e_{18} = (4 + 5 \omega) v_0 + (4 + 6 \omega) v_1 + (2 + 4 \omega) (v_2 + v_3 + v_4 + v_5). &
\end{flalign*}

\subsection{Outer normals for cusped \texorpdfstring{$P_4$}{P4}} (\autoref{cusped45678})
\begin{flalign*}
& e_i = -v_i \text{ for } 1 \leq i\leq 4, &\\
&e_i = v_0+v_{j_1}+v_{j_2} \text{ for } 5\leq i\leq 10 \text{ and } 1 \leq j_1 < j_2 \leq 4. &
\end{flalign*}

\subsection{Outer normals for cusped \texorpdfstring{$P_5$}{P5}} (\autoref{cusped45678})
\begin{flalign*}
& e_i = -v_i \text{ for } 1 \leq i\leq 5, &\\
&e_i = v_0+v_{j_1}+v_{j_2} \text{ for } 6\leq i\leq 15 \text{ and } 1 \leq j_1 < j_2 \leq 5, &\\
&e_{16} = 2v_0+v_1+v_2+v_3+v_4+v_5. &
\end{flalign*}

\subsection{Outer normals for cusped \texorpdfstring{$P_6$}{P6}} (\autoref{cusped45678})
\begin{flalign*}
& e_i = -v_i \text{ for } 1 \leq i\leq 6, &\\
&e_i = v_0+v_{j_1}+v_{j_2}  \text{ for } 7\leq i\leq 21 \text{ and } 1 \leq j_1 < j_2 \leq 6, \\
&e_i = 2v_0+v_{j_1}+\cdots+v_{j_5}  \text{ for } 22\leq i\leq 27 \text{ and } 1 \leq j_1 < j_2 < j_3 < j_4 < j_5 \leq 6. &
\end{flalign*}

\subsection{Outer normals for cusped \texorpdfstring{$P_7$}{P7}} (\autoref{cusped45678})
\begin{flalign*}
& e_i = -v_i \text{ for } 1 \leq i\leq 7, &\\
&e_i = v_0+v_{j_1}+v_{j_2} \text{ for } 8\leq i\leq 28 \text{ and }1 \leq j_1 < j_2 \leq 7, &\\
&e_i = 2v_0+v_{j_1}+\cdots+v_{j_5} \text{ for } 29\leq i\leq 49 \text{ and } 1 \leq j_1 < j_2 < j_3 < j_4 < j_5 \leq 7, &\\
&e_i = 3v_0+2v_{j_1}+v_{j_2}+\cdots+v_{j_7} \text{ for } 50\leq i\leq 56 \text{ and } 1 \leq j_1 < j_2 < j_3 < j_4 < j_5 < j_6 < j_7 \leq 7. &
\end{flalign*}

\subsection{Outer normals for cusped \texorpdfstring{$P_8$}{P8}} (\autoref{cusped45678})
\begin{flalign*}
&e_i = -v_i \text{ for } 1 \leq i\leq 8, &\\
&e_i = v_0+v_{j_1}+v_{j_2} \text{ for } 9\leq i\leq 36 \text{ and } 1 \leq j_1 < j_2 \leq 8, &\\
&e_i = 2v_0+v_{j_1}+\cdots+v_{j_5} \text{ for } 37\leq i\leq 92 \text{ and } 1 \leq j_1 < j_2 < j_3 < j_4 < j_5 \leq 8, &\\
&e_i = 3v_0+2v_{j_1}+v_{j_2}+\cdots+v_{j_7} \text{ for } 93\leq i\leq 148, &\\
&e_i = 4v_0+2(v_{j_1}+\cdots+v_{j_3})+v_{j_4}+\cdots+v_{j_7} \text{ for } 149\leq i\leq 204 &\\ & \hspace{28pt}\text{ and } 1 \leq j_1 < j_2 < j_3 < j_4 < j_5 < j_6 < j_7 \leq 7, &\\
&e_i = 5v_0+2(v_{j_1}+\cdots+v_{j_6})+v_{j_7}+v_{j_8} \text{ for } 205\leq i\leq 232, &\\
&e_i = 6v_0+3v_{j_1}+v_{j_2}+\cdots+v_{j_8} \text{ for } 233\leq i\leq 240 &\\ & \hspace{28pt}\text{ and } 1 \leq j_1 < j_2 < j_3 < j_4 < j_5 < j_6 < j_7 < j_8 \leq 8. &
\end{flalign*}

\subsection{Outer normals for cusped \texorpdfstring{$P_9$}{P9}} (\autoref{cusped9to13})
\begin{flalign*}
&e_i = -v_{i} + v_{i+1}, \text{ for } 1 \leq i\leq 8, &\\
&e_9 = -v_9, &\\
&e_{10} = v_0 + 2 v_1, &\\
&e_{11} = v_0 + v_1 + v_2 + v_3 + v_4, &\\
&e_{12} = 2 v_0 + v_1 + v_2 + v_3 + v_4 + v_5 + v_6 + v_7 + v_8 + v_9. &
\end{flalign*}

\subsection{Outer normals for cusped \texorpdfstring{$P_{10}$}{P10}} (\autoref{cusped9to13})
\begin{flalign*}
&e_i = -v_{i} + v_{i+1}, \text{ for } 1 \leq i\leq 9, &\\
&e_{10} = -v_{10}, &\\
&e_{11} = v_0 + v_1 + v_2 + v_3 + v_4, &\\ 
&e_{12} = v_0 + 2 v_1, &\\ 
&e_{13} = 2 v_0 + v_1 + v_2 +v_3 +v_4 + v_5 + v_6 + v_7 + v_8 + v_9 + v_{10}. &
\end{flalign*}

\subsection{Outer normals for cusped \texorpdfstring{$P_{11}$}{P11}} (\autoref{cusped9to13})
\begin{flalign*}
&e_i = -v_i + v_{i+1}, \text{ for } 1 \leq i\leq 7, \text{ and } i = 9, 10, &\\
&e_i = -v_i, \text{ and } i = 8, 11, &\\
&e_{12} = v_0 + v_9 + v_{10} + v_{11}, &\\
&e_{13} = v_0 + 2 v_1 + v_9, &\\
&e_{14} = v_0 + v_1 + v_2 + v_3 + v_4, &\\
&e_{15} = 2 v_0 + v_1 + v_2 + v_3 + v_4 + v_5 + v_6 + v_9 + v_{10}, &\\
&e_{16} = 2 v_0 + v_1 + v_2 + v_3 + v_4 + v_5 + v_6 + v_7 + v_8 + v_9, &\\
&e_{17} = 3 v_0 + 2 (v_1 + v_2) + v_3 + v_4 + v_5 + v_6 + v_7 + v_8 + v_9 + v_{10} + v_{11}, &\\
&e_{18} = 4 v_0 + 2 (v_1 + v_2 + v_3 + v_4 + v_5 + v_6 + v_7) + v_9 + v_{10} + v_{11}. &
\end{flalign*}

\subsection{Outer normals for cusped \texorpdfstring{$P_{12}$}{P12}} (\autoref{cusped9to13})
\begin{flalign*}
&e_i = -v_i + v_{i+1}, \text{ for } 1 \leq i\leq 10, &\\
&e_{i} = -v_{i}, \text{ for } i = 11, 12, &\\
&e_{13} = v_0 + 2 v_1 + v_{12}, &\\
&e_{14} = v_0 + v_1 + v_2 + v_3 + v_4, &\\
&e_{15} = 2 v_0 + v_1 + v_2 + v_3 + v_4 + v_5 + v_6 + v_7 + v_8 + v_{12}, &\\
&e_{16} = 2 v_0 + v_1 + v_2 + v_3 + v_4 + v_5 + v_6 + v_7 + v_8 + v_9 + v_{10}, &\\
&e_{17} = 3 v_0 + v_1 + v_2 + v_3 + v_4 + v_5 + v_6 + v_7 + v_8 + v_9 + v_{10} + v_{11} + 2 v_{12}, &\\
&e_{18} = 3 v_0 + 2 (v_1 + v_2) + v_3 + v_4 + v_5 + v_6 + v_7 + v_8 + v_9 + v_{10} + v_{11} + v_{12}, &\\
&e_{19} = 3 (v_0 + v_1) + v_2 + v_3 + v_4 + v_5 + v_6 + v_7 + v_8 + v_9 + v_{10} + v_{11}, &\\
&e_{20} = 5 v_0 + 2 (v_1 + v_2 + v_3 + v_4 + v_5 + v_6 + v_7 + v_8 + v_9 + v_{10} + v_{11} + v_{12}). &
\end{flalign*}

\subsection{Outer normals for cusped \texorpdfstring{$P_{13}$}{P13}} (\autoref{cusped9to13})
\begin{flalign*}
&e_i = -v_{i} + v_{i+1}, \text{ for } 1 \leq i\leq 12, &\\
&e_{13} = -v_{13}, &\\
&e_{14} = v_0 + v_1 + v_2 + v_3 + v_4, &\\ 
&e_{15} = v_0 + 2 v_1, &\\ 
&e_{16} = 2 v_0 + v_1 + v_2 +v _3 + v_4 + v_5 + v_6 + v_7 + v_8 + v_9 + v_{10}, &\\
&e_{17} = 3 v_0 + 3 v_1 + v_2 +v _3 + v_4 + v_5 + v_6 + v_7 + v_8 + v_9 + v_{10} + v_{11} + v_{12}, &\\
&e_{18} = 3 v_0 + 2 v_1 + 2 v_2 + v _3 + v_4 + v_5 + v_6 + v_7 + v_8 + v_9 + v_{10} + v_{11} + v_{12} + v_{13}, &\\
&e_{19} = 5 v_0 + 2 (v_1 + v_2 + v _3 + v_4 + v_5 + v_6 + v_7 + v_8 + v_9 + v_{10} + v_{11} + v_{12} + v_{13}). &
\end{flalign*}

\end{document}